\documentclass[12pt, reqno]{amsart}
\usepackage{amsmath, amsthm, amscd, amsfonts, amssymb, graphicx, color}
\usepackage[bookmarksnumbered, colorlinks, plainpages]{hyperref}
\hypersetup{colorlinks=true,linkcolor=red, anchorcolor=green, citecolor=cyan, urlcolor=red, filecolor=magenta, pdftoolbar=true}
\input{mathrsfs.sty}
\newtheorem{theorem}{Theorem}[section]
\newtheorem{lemma}[theorem]{Lemma}

\newtheorem{corollary}[theorem]{Corollary}
\theoremstyle{definition}

\newtheorem{example}[theorem]{Example}

\newtheorem{remark}[theorem]{Remark}
\numberwithin{equation}{section}

\begin{document}

\title{Extensions of the Lax--Milgram theorem to Hilbert $C^*$-modules}

\author[R. Eskandari, M. Frank, V. M. Manuilov, M. S. Moslehian]{Rasoul Eskandari$^1$, Michael Frank$^2$, Vladimir M. Manuilov$^3$, \MakeLowercase{and} Mohammad Sal Moslehian$^4$}
\address{$^1$Department of Mathematics, Faculty of Science, Farhangian University, Tehran, Iran.}
\email{eskandarirasoul@yahoo.com}

\address{$^2$Hochschule f\"ur Technik, Wirtschaft und Kultur (HTWK) Leipzig, Fakult\"at Informatik und Medien, PF 301166, D-04251 Leipzig, Germany.}
\email{michael.frank.leipzig@gmx.de}

\address{$^3$ Department of Mechanics and mathematics, Moscow State University, Moscow, 119992, Russia.} 
\email{manuilov@mech.math.msu.su}

\address{$^4$Department of Pure Mathematics, Center of Excellence in Analysis on Algebraic Structures (CEAAS), Ferdowsi University of Mashhad, P. O. Box 1159, Mashhad 91775, Iran.}
\email{moslehian@um.ac.ir}

\subjclass[2010]{46L08, 46L05, 46C05.}

\keywords{Lax--Milgram theorem; sesqulinear form; Hilbert $C^*$-module; self-duality.}

\begin{abstract}
We present three versions of the Lax--Milgram theorem in the framework of Hilbert $C^*$-modules, two for those over $W^*$-algebras and one for those over $C^*$-algebras of compact operators. It is remarkable that while the Riesz theorem is not valid for certain Hilbert $C^*$-modules over $C^*$-algebras of compact operators, our Lax--Milgram theorem turns out to be valid for all of them. We also give several examples to illustrate our results, in particular, we show that the main theorem is not true for Hilbert modules over arbitrary $C^*$-algebras.
\end{abstract}

\maketitle

\section{Introduction and preliminaries}

The Lax--Milgram theorem, due to P. D. Lax and A. N. Milgram \cite{LM}, can be considered as a kind of representation theorem for bounded linear functionals on a (real) Hilbert space, in other words, it extends the Riesz representation theorem from the inner product to a bilinear form under certain conditions. It asserts that if $\varphi$ is a bounded bilinear form on a Hilbert space $\mathscr{H}$ such that $\varphi(x,x)\geq M\|x\|$ for all $x\in \mathscr{H}$, then for every bounded linear functional $f$ on $\mathscr{H}$, there exists a unique element $u$ such that $f(x)=\varphi(x,u)$ for all $x\in \mathscr{H}$ and $\|u\|\leq \|f\|/M$. It has several applications in various disciplines such as partial differential equations, integral operators, fractional differential equations, differential operators, system of boundary value problems, linear system of equations, etc; see for example \cite{KY, LKL}. This theorem has been generalized by several mathematicians in linear and nonlinear forms; see for example \cite{RAM, SAI, DY, RUI}. A version of the Lax-Milgram theorem for Hilbert $\mathbb{C}$-modules and $\mathbb{C}$-sesquilinear forms is given in \cite{GV}. In this note, we adopt the version of the Lax--Milgram theorem presented in \cite{DY} and generalize the theorem in the setting of Hilbert $C^*$-modules over $W^*$-modules and over $C^*$-algebras of compact operators.

Let us fix our notation and terminology. The notion of an inner product $C^*$-module is a natural generalization of that of an inner product space arising under the replacement of the field of scalars $\mathbb{C}$ by a $C^*$-algebra. For an inner product $C^*$-module $(\mathscr{X}, \langle\cdot,\cdot\rangle)$ over a $C^*$-algebra $\mathscr{A}$, we denote the induced norm by $\|x\|:=\|\langle x, x\rangle\|^{\frac{1}{2}}$ and the positive square root of $\langle x, x\rangle$ by $|x|$ for $x\in \mathscr{X}$. A Hilbert $\mathscr A$-module $\mathscr X$ is full if the linear span of all $\langle x,y\rangle$, $x,y\in\mathscr X$, is dense in $\mathscr A$. We say that a closed submodule $\mathscr{Y}$ of a Hilbert $C^*$-module $\mathscr{X}$ is orthogonally complemented if $\mathscr{X}=\mathscr{Y}\oplus \mathscr{Y}^{\perp}$, where $\mathscr{Y}^{\perp}=\{ x\in \mathscr{X} : \langle x,y\rangle=0$ for all $ y\in \mathscr{Y}\}$. 

Note that the theory of inner product $C^*$-modules is quite different from that of inner product spaces. Indeed, some fundamental properties of inner product spaces are no longer valid in inner product $C^*$-modules in their full generality. For example, not any closed submodule of an inner product $C^*$-module is complemented, and a bounded $C^*$-linear map on an inner product $C^*$-module may not have an adjoint operator. 

Let $\mathscr{X}^{\prime}$ stand for the set of all bounded $\mathscr{A}$-linear maps from $\mathscr{X}$ into $\mathscr{A}$. It is easy to verify that $\mathscr{X}^{\prime}$ becomes a right $\mathscr{A}$-module if we set $(\rho+\lambda\tau)(x)=\rho(x)+\overline{\lambda}\tau(x)$ and $(\tau b)(x) =b^*\tau (x)$ for $\tau \in \mathscr{X}^{\prime}, b\in \mathscr{A}, \lambda\in\mathbb{C}$ and $x\in \mathscr{X}$. 
Each $x\in \mathscr{X}$ gives rise to a map $\hat{x}\in \mathscr{X}^{\prime}$ defined by $\hat{x}(y)=\langle x,y\rangle$ for $y\in \mathscr{X}$. The map $x \mapsto \hat{x}$ is an isometric ${\mathscr A}$-linear map. Hence we can identify ${\mathscr X}$ with $\hat{\mathscr{X}}:=\{\hat{x}: x\in \mathscr{X}\}$ as a submodule of $\mathscr{X}^{\prime}$. A module $\mathscr{X}$ is called self-dual if $\hat{\mathscr{X}}=\mathscr{X}^{\prime}$. For example. when $\mathscr{A}$ is unital, every algebraically finitely generated projective $\mathscr{A}$-module (in particular, $\mathscr{A}$ itself) is self-dual. 

An interesting question is whether $\langle\cdot,\cdot\rangle$ on $\mathscr{X}$ can be extended to an $\mathscr{A}$-valued inner product on $\mathscr{X}^{\prime}$? W. L. Paschke \cite{Pa} showed that it can be done if $\mathscr{A}$ is a $W^*$-algebra, that is, a $C^*$-algebra for which there is a Banach space $\mathscr{A}_*$ such that its Banach dual is $\mathscr{A}$. The space  $\mathscr{A}_*$ is indeed a closed subspace of the Banach dual of $\mathscr{A}$ and its elements are called normal (bounded linear functionals). Huaxin Lin \cite{LIN} proved that if $\mathscr{A}$ is monotone complete, then the inner product of $\mathscr{X}$ can be extended to an inner product on $\mathscr{X}^\prime$ turning $\mathscr{X}^\prime$ into a self-dual Hilbert $\mathscr{A}$-module.

We are interested in the case of Hilbert $C^*$-modules over $W^*$-algebras. 
Following \cite{Pa}, let $f$ be a positive linear functional on $\mathscr{A}$. It is easy to see that $f(\langle \cdot, \cdot\rangle)$ is a semi-inner product
on $\mathscr{X}$ and that $\mathscr{N}_f = \{x:f(\langle x,x\rangle)=0\}$ is a linear subspace of $\mathscr{X}$. It follows
that $\mathscr{X}/\mathscr{N}_f$ is an inner product space equipped with the inner product $(\cdot, \cdot)_f$ defined by
$(x + \mathscr{N}_f, y + \mathscr{N}_f)_f := f(\langle y,x\rangle )$ for $x, y \in \mathscr{X}$. We denote by $\mathscr{H}_f$ the Hilbert space completion
of $\mathscr{X}/\mathscr{N}_f$ and write $\|\cdot\|_f$ for the norm on $\mathscr{H}_f$ obtained from its inner product $(\cdot,\cdot)_f$. For $\tau\in \mathscr{X}^\prime$, the map $x+\mathscr{N}_f \mapsto f(\tau(x))$ is a bounded linear functional with norm less than or equal to $\|\tau\|\,\|f\|^{1/2}$. Hence, there is a unique vector $\tau_f\in \mathscr{H}_f$ such that $\|\tau_f\|_f\leq \|\tau\|\|f\|^{1/2}$ and 
\begin{eqnarray}\label{Pas1}
(x+\mathscr{N}_f,\tau_f)_f=f(\tau (x))
\end{eqnarray}
for all $x\in \mathscr{X}$. Now we are ready to state the main result of Paschke.

\begin{theorem}\cite[Theorem 3.2]{Pa}\label{th1}
Let $\mathscr{X}$ be a pre-Hilbert $C^*$-module over a $W^*$-algebra $\mathscr{A}$. The $\mathscr{A}$-valued inner product $\langle \cdot,\cdot\rangle$ can be exteded to $\mathscr{X}^{\prime}\times \mathscr{X}^{\prime}$ in such a way as to make $\mathscr{X}^\prime$ into a self-dual Hilbert $\mathscr{A}$-module. In particular, the extended inner product satisfies \begin{eqnarray}\label{Pas2}
\langle \hat{x},\tau\rangle=\tau (x)
\end{eqnarray}
and 
\begin{eqnarray}\label{Pas3}
f(\langle\tau,\rho\rangle) =(\tau_f,\rho_f)_f
\end{eqnarray}
for all $ x\in \mathscr{X}, \tau, \rho\in \mathscr{X}^{\prime}$, and all normal positive linear functionals $f$ on $\mathscr{A}$.
\end{theorem}
We also need the following lemma.
\begin{lemma}\cite[Proposition 3.8]{Pa}\label{l3}
Let $\mathscr{X}$ be a self-dual Hilbert $C^*$-module  over a $W^*$-algebra. Then $\mathscr{X}$ is a conjugate space.
\end{lemma}

$C^*$-algebras $\mathscr{A}$ of compact operators are those which admit a faithful $*$-represen\-tation in the set of compact operators on some Hilbert space. By \cite[{\S} I.4 and Theorem~1.4.5]{Arw}, such $C^*$-algebras have a canonical decomposition as $A = c_0-\Sigma_\alpha \oplus K(H_\alpha)$, where $K(H_\alpha)$ denotes the $C^*$-algebra of all compact operators on some Hilbert space $H_\alpha$, and the $c_0$-sum is either a finite block-diagonal sum or a block-diagonal sum with a $c_0$-convergence condition on the $C^*$-algebra components $K(H_\alpha)$. The $c_0$-sum may possess arbitrary cardinality. Their multiplier algebras $M({\mathscr{A}})$ can be characterized as $M({\mathscr{A}}) = l_\infty-\Sigma_\alpha B(H_\alpha)$, where $B(H_\alpha)$ are the $W^*$-algebras of all bounded linear operators on the corresponding Hilbert spaces $H_\alpha$ and the norm of the components is globally bounded for every element by some element-specific constant. They are special $W^*$-algebras of type I, cf. \cite[Theorem 3.6]{Fr08}. 

Hilbert $C^*$-modules over $C^*$-algebras of compact operators possess many properties in common with Hilbert spaces, except that most of them are not self-dual.

\begin{theorem} \cite[Theorem 3.1 and Corollary 3.2]{Fr08}, \cite{Schw}
Let $\mathscr{A}$ be a $C^*$-algebra. The following conditions are equivalent:
\begin{itemize}
\item[(i)] $\mathscr{A}$ is a $C^*$-algebra of compact operators.
\item[(ii)] For every Hilbert $\mathscr{A}$-module $\mathscr{X}$ every Hilbert $\mathscr{A}$-submodule ${\mathscr{Y}} \subseteq \mathscr{X}$ is automatically ortogonally complemented in $\mathscr{X}$, i.e. $\mathscr{Y}$ is an orthogonal summand of $\mathscr{X}$. 
\item[(iii)] For every pair of Hilbert $\mathscr{A}$-modules $\mathscr{X}$, $\mathscr{Y}$ and every bounded $\mathscr{A}$-linear map $T: {\mathscr{X}} \to \mathscr{Y}$ there exists an adjoint bounded $\mathscr{A}$-linear map $T^*: {\mathscr{Y}} \to \mathscr{X}$. 
\end{itemize}
\end{theorem}

The following construction of a multiplier module and its properties are from \cite{BaGu04,BaGu06}:
For every Hilbert $\mathscr{A}$-module $\mathscr{X}$ there exists a Hilbert $M({\mathscr{A}})$-module $M({\mathscr{X}})$ containing $\mathscr{X}$ as an ideal submodule associated with the ideal $\mathscr{A}$ in $M({\mathscr{A}})$, i.e. ${\mathscr{X}} = M({\mathscr{X}}) {\mathscr{A}}$. Moreover, ${\mathscr{X}} = \{ x \in M({\mathscr{X}}) \, :\, \langle x,v\rangle \in {\mathscr{A}} \,\, {\rm for} \, {\rm any} \, v \in M({\mathscr{X}}) \}$. The extended module $M({\mathscr{X}})$ is called the multiplier module of $\mathscr{X}$. It can be naturally identified with the set of all bounded adjointable $\mathscr{A}$-linear maps from $\mathscr{A}$ into $\mathscr{X}$. If $\mathscr{Y}$ is another Hilbert $\mathscr{A}$-module, then every bounded adjointable $\mathscr{A}$-linear operator $T: {\mathscr{X}} \to \mathscr{Y}$ has a unique extension $T_M: M({\mathscr{X}}) \to M({\mathscr{Y}})$ of the same operator norm such that the canonical embeddings of $\mathscr{X}$ and $\mathscr{Y}$ into $M({\mathscr{X}})$ and $M({\mathscr{Y}})$, respectively, are preserved.  Moreover, $(T_M)^*=(T^*)_M$. Moreover, for any bounded $M(\mathscr X)$-linear operator $T: M({\mathscr X}) \to M({\mathscr Y})$, the restriction of $T$ to ${\mathscr X}$ is a map into ${\mathscr Y}$. 

Applying this construction to Hilbert $\mathscr{A}$-modules over $C^*$-algebras $\mathscr{A}$ of compact operators we see that $M({\mathscr{X}})$ is a self-dual Hilbert $W^*$-module over the $W^*$-algebra $M({\mathscr{A}})$, cf.~\cite[Theorem 5.5]{Fr99}. So, this gives us an idea how to compensate the possible lack of self-duality of $\mathscr{X}$ switching to its extension $M({\mathscr{X}})$. 

We refer the reader to \cite{LAN, MT} for more information on the basic theory of Hilbert $C^*$-modules.

\section{The Lax--Milgram Theorem extended to self-dual Hilbert $W^*$-modules}

We start our work with a useful observation concerning the self-dual Hilbert $C^*$-modules. We notice that in the proof of the implication (i)$\Longrightarrow$(ii) we do not use the assumption of the self-duality of $\mathscr{X}$.
\begin{lemma}
Let $\mathscr{X}$ be a self-dual Hilbert $C^*$-module over a $C^*$ algebra $\mathscr{A}$ and let $\mathscr{Y}$ be a closed submodule of $\mathscr{X}$. Then the following assertions are equivalent:
\begin{itemize}
\item[(i)] $\mathscr{Y}$ is self-dual.
\item[(ii)] $\mathscr{Y}$ is an orthogonally complemented submodule of $\mathscr{X}$.
\end{itemize}
\end{lemma}
\begin{proof}
(i)$\Longrightarrow$(ii). Suppose $\mathscr{Y}$ is self-dual. For any $x\in \mathscr{X}$ we have $\hat{x}|_\mathscr{Y}\in \mathscr{Y}^\prime$ . Hence there is $y_0\in \mathscr{Y}$ such that $\hat{x}|_\mathscr{Y}=\hat{y_0}$. Hence, $\langle x-y_0,y\rangle=0$ for all $y\in \mathscr{Y}$. Therefore $x-y_0\in \mathscr{Y}^\perp$. Hence $x=y_0+(x-y_0)\in \mathscr{Y}+\mathscr{Y}^\perp$. This shows that (ii) holds.\\
(ii)$\Longrightarrow$(i). Suppose $\mathscr{Y}$ is orthogonally complemented and $\tau\in \mathscr{Y}^\prime$. We can extend $\tau$ to $\tilde{\tau}\in \mathscr{X}^{\prime}$ by $\tilde{\tau}(y_1+y_2)=\tau(y_1)$ for $y_1\in \mathscr{Y},y_2\in \mathscr{Y}^\perp$. Since $\mathscr{X}$ is self-dual, there is $x\in \mathscr{X}$ such that $\tilde{\tau}(z)=\langle x, z\rangle$ for all $z\in \mathscr{X})$. Hence $\langle x,y\rangle =\tilde{\tau}(y)=0$ for all $y\in \mathscr{Y}^\perp$. Therefore $x\in \mathscr{Y}^{\perp\perp}=\mathscr{Y}$.
\end{proof}
Let $\mathcal P\mathcal{S}(\mathscr{A})$ denote the set of all pure states on $\mathscr A$. We are ready to state our first result.

\begin{theorem}\label{l2}
Let $\mathscr{X},\mathscr{Y}$ be Hilbert $C^*$-modules over a $W^*$-algebra $\mathscr{A}$ and let $\mathscr{X}$ be self-dual. Let $B:\mathscr{X}\times \mathscr{Y}\to \mathscr{A}$ be a bounded $\mathscr{A}$-sesqulinear form (i.e. $\mathscr{A}$-linear on $\mathscr{Y}$ and anti-$\mathscr{A}$-linear on $\mathscr X$) such that there exists $c>0$ and $k>0$ such that
\begin{enumerate}
\item
for any $f\in\mathcal P\mathcal S(\mathscr A)$ and for any $x\in\mathscr X$ there exists $y\in\mathscr Y$ with $\|y\|=1$ such that $f(|y|)\geq k$ and
\begin{equation}\label{main}
|f(B(x,y))|\geq cf(|x|)f(|y|);
\end{equation}
\item
for any $f\in\mathcal P\mathcal S(\mathscr A)$ and for any $y\in\mathscr Y$ there exists $x\in\mathscr X$ with $\|x\|=1$ such that $f(|x|)\geq k$ and
\begin{equation}\label{main2}
|f(B(x,y))|\geq cf(|x|)f(|y|).
\end{equation}
\end{enumerate}
Then for any $\tau\in \mathscr{Y}^\prime$ there is a unique element $x\in \mathscr{X}$ such that $\|x\|\leq \frac{\|\tau\|}{c}$ and
\[
B(x,y)=\tau(y)\qquad ( y\in \mathscr{Y}).
\]
\end{theorem}
\begin{proof}
Let $T:\mathscr{X}\to \mathscr{Y}^\prime$ be defined by
$$T(x)(y)=B(x,y)\in\mathscr A\qquad (y\in \mathscr{Y})\,.$$
It is easy to see that $T$ is a bounded $\mathscr{A}$-linear map. We shall show that $T$ is adjointable: For $\tau\in \mathscr{Y}^\prime,$ we define $\phi_\tau\in X'$ by $\phi_\tau (x)=\langle \tau,Tx\rangle$. Due to $\mathscr{X}^{\prime}$ is self-dual, there exits $z_\tau\in \mathscr{X}$ such that $\langle \tau,Tx\rangle =\phi_\tau(x)=\langle x, z_\tau\rangle$. Now we define $S:\mathscr{Y}^\prime\to \mathscr{X}$ by $S(\tau)=z_\tau$. We have $ \langle \tau,Tx\rangle =\langle z_\tau,x\rangle =\langle S\tau,x\rangle$, so that $T^*=S$.

Next, we show that the range $\mathcal{R}(T)$ of $T$ is closed. It follows from (\ref{main}) that $T$ is bounded below. Indeed, let $x\in\mathscr X$. Then take $f_0\in\mathcal P\mathcal S(\mathscr A)$ such that $f_0(|x|)=\|x\|$, and take $y_0\in\mathscr Y$ such that $\|y_0\|=1$, $f_0(|y_0|)\geq k$, and (\ref{main}) holds for $y=y_0$. Then

\begin{eqnarray*}
\|Tx\|&=& \sup_{\tau\in\mathscr{Y}^\prime, \|\tau\|=1}\|\langle\tau,Tx\rangle\|\\
&\geq&\sup_{f\in\mathcal P\mathcal S(\mathscr A),\tau\in\mathscr{Y}^\prime,\|\tau\|=1}|f(\langle \tau,Tx\rangle)|\\
&\geq&\sup_{f\in\mathcal P\mathcal S(\mathscr A),y\in\mathscr{Y},\|y\|=1}|f(\langle \hat{y},Tx\rangle)|\\
&=&\sup_{f\in\mathcal P\mathcal S(\mathscr A),y\in\mathscr{Y},\|y\|=1}|f(Tx(y))|\qquad\qquad (\mbox{by~} \eqref{Pas2})\\
&=&\sup_{f\in\mathcal P\mathcal S(\mathscr A),y\in\mathscr{Y},\|y\|=1}|f(B(x,y))|\\
&\geq&|f_0(B(x,y_0))|\geq cf_0(|x|)f_0(|y_0|)\\
&\geq&cf_0(|x|) k=ck\cdot\|x\|.
\end{eqnarray*}

Therefore, if $T(x_n)\to \tau$ as $n\to \infty$, then  $x_n\to z$ for some $z\in \mathscr{X}$. Hence $Tx_n\to Tz.$ This shows that $\mathcal{R}(T)$ is closed.

By \cite[Theorem 3.2]{LAN} we get
 \[
 \mathscr{Y}^\prime=T(\mathscr{X})\oplus T(\mathscr{X})^\perp\,.
 \]
 
We claim that $T(\mathscr{X})^\perp=0$. Indeed, if $\tau\in T(\mathscr{X})^\perp$, then for all $x\in \mathscr{X}$
 \begin{align*}
 \langle Tx, \tau\rangle=0\,.
 \end{align*}
 Let $f\in \mathcal P\mathcal{S}(\mathscr{A})$ be normal. We have
 \begin{align*}
 \left((Tx)_f,\tau_f\right)_f=f\left(\langle Tx, \tau\rangle\right) =0\quad (x\in \mathscr{X})\qquad\qquad (\mbox{by~} \eqref{Pas3})\,.
 \end{align*}
 Since $\tau_f\in \mathscr{H}_f$ there is $\{y_n\}\subset \mathscr{Y}$ (depending on $f$) such that $y_n+N_f\to \tau_f$. Therefore
 \begin{align}\label{eq1}
 \lim_{n}\left((Tx)_f, y_n+N_f\right)_f=0 \quad (x\in \mathscr{X}).
 \end{align}
Let $L$ be the norm closure of the subspace $\{(Tx)_f:x\in\mathscr X\}$ of $\mathscr H_f$. Decompose $\eta_n=y_n+N_f$ as $\eta_n=\eta_n^0+\eta_n^\perp$, where $\eta_n^0\in L$, $\eta_n^\perp\in L^\perp$. Using a standard limit argument, referring to the bondedness of $\{\eta_n^0\}$, and employing (\ref{eq1}), we infer that $\lim_{n\to\infty}\|\eta_n^0\|=0$.  
By assumption, for each $y_n$ there exists $x_n\in\mathscr X$ such that $\|x_n\|=1$, $f(|x_n|)\geq k$, and (\ref{main2}) holds.   
Note that
$|\left((Tx_n)_f,\eta_n\right)_f|=|\left((Tx_n)_f,\eta_n^0\right)_f|\leq \|T\|\|\eta_n^0\|$,
whence
 \begin{align}\label{eq1'}
 \lim_{n}\left((Tx_n)_f, y_n+N_f\right)_f=0.
 \end{align}

Using (\ref{main2}), we get
 \begin{align}\label{eq2}
ck\cdot f(|y_n|)&\leq cf(|x_n|)f(|y_n|)\leq |f(B(x_n,y_n))|=|f((Tx_n)(y_n))|\nonumber\\
&=|(y_n+N_f,(Tx_n)_f)_f|\qquad(\mbox{by~} \eqref{Pas1})\,.
 \end{align}
 It follows from \eqref{eq1'} and \eqref{eq2} that
$\lim_nf(|y_n|)=0$, hence $\lim_{n}f(\langle y_n,y_n\rangle)=0$. Thus $\tau_f=0$ and so $f(\tau(y))=0$
 for all normal functionals $f\in \mathcal P\mathcal{S}(\mathscr{A})$ and all $y\in \mathscr{Y}$. Hence $\tau=0$. Thus $\mathscr{Y}=T(\mathscr{X})$.

Note that $B$ is nondegenerate with respect to the first variable, that is, for each $x\in \mathscr{X}\backslash\{0\}$, there exists $y\in \mathscr{Y}$ such that $B(x,y)\neq 0$. Indeed, if $x\neq 0$, then there is $f\in\mathcal P\mathcal S(\mathscr A)$ such that $f(|x|)\neq 0$. Non-degeneracy of $B$ follows now from (\ref{main}). Therefore, if $B(x_1,y)=B(x_2,y)$ for any $y$ then $x_1=x_2$.

Let $\tau\in \mathscr{Y}^\prime$ and let $x\in \mathscr{X}$ such that 
\[
B(x,v)=\tau(v)\quad(v\in \mathscr{Y})\,.
\]
Thus,
\[
c\|x \|\leq \sup_{\|v\|=1}\|B(x,v)\|=\sup_{\|v\|=1}\|\tau(v)\|=\|\tau\|\,.
\]
Hence $\|x\|\leq \frac{\|\tau\|}{c}$.
\end{proof}

\begin{corollary}\label{cor-uniform}
Let $\mathscr{X},\mathscr{Y}$ be full Hilbert $C^*$-modules over a $W^*$-algebra $\mathscr{A}$, and let $\mathscr{X}$ be self-dual. Let $B:\mathscr{X}\times \mathscr{Y}\to \mathscr{A}$ be a bounded $\mathscr{A}$-sesqulinear form (i.e., $\mathscr{A}$-linear on $\mathscr{Y}$ and anti-$\mathscr{A}$-linear on $\mathscr X$) such that there exists $c>0$ such that
$$
|f(B(x,y))|\geq cf(|x|)f(|y|)
$$
for any $f\in\mathcal P\mathcal S(\mathscr A)$, and for any $x\in\mathscr X$, $y\in\mathscr Y$.
Then for any $\tau\in \mathscr{Y}^\prime$ there is a unique element $x\in \mathscr{X}$ such that $\|x\|\leq \frac{\|\tau\|}{c}$ and
\[
B(x,y)=\tau(y)\qquad ( y\in \mathscr{Y}).
\]
\end{corollary}
\begin{proof}
If the assumption holds then it suffices to show that there exists $k>0$ such that for any $f\in\mathcal P\mathcal S(\mathscr A)$, there exist $x\in\mathscr X$ and $y\in\mathscr Y$ of norm 1 such that $f(|x|)>k$, $f(|y|)>k$. Then the assumptions of Theorem \ref{l2} hold true.

As $\mathscr X$ is full, there exists $m\in\mathbb N$ and $x_1,\ldots,x_m\in\mathscr{X}$ such that $\sum_{i=1}^m\langle x_i,x_i\rangle=1$ (cf. \cite[Theorem 2.1]{Brown} and  \cite[Lemma 2.4.3]{MT}). Therefore, for any $f\in\mathcal P\mathcal S(\mathscr A)$ there exists $x\in\mathscr{X}$ such that $\|x\|\leq 1$ and $f(\langle x,x\rangle)\geq \frac{1}{m}$. Then $f(|x|)\geq f(|x|^2)\geq\frac{1}{m}$. A similar argument works for $y$.

\end{proof}

The following example shows that the conditions of Corollary \ref{cor-uniform} are too restrictive.

\begin{example}
Let $\mathscr A=M_2$ be the algebra of two-by-two matrices, $\mathscr X=\mathscr Y=\mathscr A$, $B(x,y)=x^*y$. Then $B$ satisfies the conditions of Theorem \ref{l2}, but does not satisfy those of Corollary \ref{cor-uniform}. Indeed, given $f\in\mathcal P\mathcal S(\mathscr A)$ and $x\in M_2$, take a polar decomposition $x=uh$ with a unitary $u$ and positive $h$, and set $y=u$. Then $f(|y|)=f(1)=1$, and 
$$
|f(B(x,y))|=|f(hu^*u)|=f(h)=f(|x|)=f(|x|)f(|y|).
$$
On the other hand, let $x=\frac{1}{2}\left(\begin{smallmatrix}1&1\\1&1\end{smallmatrix}\right)$, $y=\frac{1}{2}\left(\begin{smallmatrix}1&-1\\-1&1\end{smallmatrix}\right)$, and let $f$ be the pure state given by $f\left(\left(\begin{smallmatrix}a_{11}&a_{12}\\a_{21}&a_{22}\end{smallmatrix}\right)\right)=a_{11}$. Then $x^*y=0$, hence $f(B(x,y))=0$, while $f(|x|)=f(|y|)=\frac{1}{2}$. 

\end{example}

\begin{corollary}\label{LM_Wstar}
Let $\mathscr{X}$ and $\mathscr{Y}$ be self-dual Hilbert $C^*$-modules over a $W^*$-algebra $\mathscr{A}$. Suppose that
\[
B:\mathscr{X}\times \mathscr{Y}\to \mathscr{A}
\]
is a map satisfying the conditions of Theorem \ref{l2}.
Then there exists a bounded adjointable invertible $\mathscr{A}$-linear operator $T: {\mathscr{X}} \to \mathscr{Y}$ such that $B(x,y) = \langle T(x),y \rangle$ for any $x \in \mathscr{X}$, any $y \in \mathscr{Y}$. 

\end{corollary}
\begin{proof}
The operator $T:\mathscr X\to\mathscr Y'$ was constructed in the proof of Theorem \ref{l2}. As $\mathscr Y$ is self-dual, we are done.

\end{proof}

\begin{example} \label{ex_inn_prod}
Let $\mathscr{A}$ be a $W^*$-algebra and $\mathscr X$ be self-dual and full, let $\mathscr Y=\mathscr X$, and let $B(x,y)=\langle x,y\rangle$. Then $B$ satisfies the conditions of Theorem \ref{l2}. Indeed, let $x\in\mathscr X$. By \cite[Proposition 3.11]{Pa}, there exists $z\in\mathscr X$ such that $p=\langle z,z\rangle$ is the range projection of $|x|$ and $x=z\cdot|x|$. The Hilbert module $\mathscr Z$ generated by $z$ is projective, hence  $\mathscr X=\mathscr Z\oplus\mathscr Z^\perp$. If $\mathscr X$ is full, then so is $\mathscr Z^\perp$ as a module over $(1-p)\mathscr A(1-p)$. Therefore, there exists $k\in\mathbb N$ and $z_1,\ldots,z_k\in\mathscr Z^\perp$ such that $\sum_{i=1}^k|z_i|^2=1-p$ (by \cite[Theorem 2.1]{Brown}; cf. \cite[Lemma 2.4.3]{MT}). Hence,  we can find $z'\in\mathscr Z^\perp$ such that $f(|z'|)> k$. Take $y=z+z'$. Then $f(|z+z'|)=f\left(\left(|z|^2+|z'|^2\right)^{1/2}\right)\geq f(|z'|)> k$, and $B(x,y)=\langle z\cdot|x|,z+z'\rangle=|x|\langle z,z\rangle=|x|$, so $|f(B(x,y))|=f(|x|)\geq f(|x|)f(|y|)$ (as $\|y\|=1$) for any $f\in\mathcal P\mathcal S(\mathscr A)$. The second condition of Theorem \ref{l2} can be verified in the same way. It is easy to see that, in this example, $T$ is the identity operator.
\end{example}

The following example shows that our result cannot be transferred from $W^*$-algebras to general $C^*$-algebras, even in the commutative case.
\begin{example}
Let $\mathscr{A}=C([0,1])$ be the $C^*$-algebra of continuous functions on the unit interval $[0,1]$ which is equipped with the standard topology derived from the real numbers, and let $\mathscr{J}:=\{u\in \mathscr{A}: u(0)=0\} $, which is a closed ideal of $\mathscr{A}$. Take $\mathscr{X}=\mathscr{A}$ and $\mathscr{Y}=\mathscr{J}$. Let $B:\mathscr{X}\times \mathscr{Y}\to \mathscr{A}$ be defined by $B(u,v):=\overline{u}v$ for $u\in \mathscr{X}, v\in \mathscr{Y}$. It is easy to see that $B$ is a bounded $\mathscr{A}$-sesqulinear form.
Let $f\in \mathcal{PS}(\mathscr{A})$. Then 
\begin{align*}
|f(B(u,v))|&=|f(\overline{u}v)|=f(|u|)f(|v|)
\end{align*}
for all $u\in\mathscr{X}, v\in \mathscr{Y}$. Utilizig the fact that for any state $f$,  $\sup\{f(a): \|a\|=1, a\geq 0\} = 1$, $B$ satisfies the conditions of Theorem \ref{l2} with $c=1$ and any $k<1$, except that $\mathscr A$ is not a $W^*$-algebra. Note that $\mathscr{Y}^\prime=C_b(0,1]$ consists of all bounded continuous functions on the half-open unit interval $(0,1]$, and if $\tau\in\mathscr{Y}^\prime$, $v\in\mathscr{Y}$, then $\tau(v)=\tau\cdot \overline v$. If $\tau$ has no limit at $0$, then there is no $u\in\mathscr X$ such that $\tau(v)=u\cdot \overline v$ for any $v\in\mathscr{Y}$. For example, take $\tau(t)=\sin\frac{1}{t}$, $t\in(0,1]$. Then $\tau\in \mathscr{Y}^\prime$. If $\tau(y)=\langle x,y\rangle$ for some $x\in C[0,1]$, then $x(t)=\tau(t)$ for any $t\in(0,1]$ since this should hold on supports of all $y\in \mathscr{Y}=C_0(0,1]$. But there is no function $x(t)$ in $C[0,1]$ such that $x(t)=\tau(t)$ for any $t\in(0,1]$. 
\end{example}

\begin{theorem}\label{tha}
Let $\mathscr{X}$ and $\mathscr{Y}$ be Hilbert $C^*$-modules over a $W^*$-algebra $\mathscr{A}$ and let $\mathscr{X}$ be self-dual. Let $\{\mathscr{X}_\lambda\}_{\lambda\in \Lambda}$ be a family of closed orthogonally complemented submodules of $\mathscr{X}$, let $\{\mathscr{Y}_\lambda\}_{\lambda\in \Lambda}$ be an upwards directed family of closed submodules of $\mathscr{Y}$ and let $\mathscr{V}=\cup \mathscr{Y}_\lambda$. Suppose that
\[
B:\mathscr{X}\times \mathscr{V}\to \mathscr{A}
\]
is a map satisfying:
\begin{itemize}
\item[(a)]$B_\lambda=B|_{\mathscr{X}_\lambda\times \mathscr{Y}_\lambda}$ satisfies the conditions of Theorem \ref{l2} with the same constants $c$ and $k$ for all $\lambda\in\Lambda$;
\item[(b)]$B(.,v)$ is a bounded $\mathscr{A}$-linear functional on $\mathscr{X}$ for all $v\in \mathscr{V}$
 \end{itemize}
 Then for each $\tau\in \mathscr{V}',$ there is $x\in \mathscr{X}$ such that $\|x\|\leq \frac{\|\tau\|}{c}$ and
 \[
 B(x,v)=\tau(v)
 \]
 for all $v\in \mathscr{V}$.
\end{theorem}
\begin{proof}

Let $\tau\in \mathscr{V}'$. Set $\tau_\lambda=\tau|_{\mathscr{Y}_\lambda}$. Then $\tau_\lambda\in \mathscr{Y}_\lambda'$. By Theorem \ref{l2}, there is an element $x_\lambda\in\mathscr X_\lambda$ such that $\|x_\lambda\|\leq \frac{\|\tau_\lambda\|}{c}$ and
\[
B_\lambda (x_\lambda, y)=\tau_\lambda(y)\qquad ( y\in Y_\lambda)\,.
\]
It is easy to see that $c\|x_\lambda\|\leq \|\tau\|$. So $\{x_\lambda\}$ is a bounded net. It follows from Lemma \ref{l3} that $\mathscr{X}$ is a conjugate space. Applying the Banach--Alaoglu theorem, and by passing to a subnet if necessary, we may assume that there is $z\in \mathscr{X}$ such that $x_{\lambda}\to z$ in the weak* topology. We have
 \[
\|z\|\leq \liminf_{n\to \infty}\|x_\lambda\|\leq\liminf_{n\to \infty}\frac{\|\tau_\lambda\|}{c}\leq \frac{\|\tau\|}{c}
\]
Now we prove that $B(z,v)=\tau(v)$ for all $v\in \mathscr{V}$.
To this end, let $v\in \mathscr{V}$. There is $\lambda_0\in \Lambda$ such that $v\in \mathscr{Y}_{\lambda_0}$ for all $\lambda\geq \lambda_0$ and so that
\begin{align}\label{eqw}
B(x_{\lambda},v)=\tau_{\lambda}(v)=\tau(v) \,.
\end{align}
On the other hand, if $\rho(x):=B(x,v)^*$, then $\rho\in \mathscr{X}^\prime$. Since $\mathscr{X}$ is self-dual, $\rho(\cdot)=\langle u,\cdot\rangle $ for some $u\in \mathscr{X}$. Hence $B(.,v)=\langle\cdot,u\rangle$. Now suppose $f\in \mathcal P\mathcal{S}(\mathscr{A})$ is normal. Utilizing \cite[Remark 3.9]{Pa}, we get 
\[
f(B(x_\lambda,v))=f(\langle x_\lambda, u\rangle)\to f(\langle z,u\rangle )
\]
Also \eqref{eqw} gives us $f(B(x_\lambda,v))\to f(\tau (v))$ hence by  uniqueness of limit, we get $f(B(z,v))=f(\tau(v))$ for all normal functionals $f\in \mathcal P\mathcal{S}(\mathscr{A})$. Thus $B(z,v)=\tau (v)$.
\end{proof}

The following fact can be easily derived using the self-duality of the Hilbert $C^*$-module and the positivity of the $C^*$-valued inner product (cf. Example \ref{ex_inn_prod}).

\begin{remark}
Let $\mathscr A$ be a $W^*$-algebra and $\mathscr X$ be a full Hilbert $\mathscr A$-module. Let $T$ be a positive bounded invertible $\mathscr A$-linear operator on $\mathscr X$. Then the sesqilinear form $B_T(x,y)=\langle T(x),y \rangle$ satisfies the conditions of Theorem \ref{l2}. Indeed, by Example \ref{ex_inn_prod} for any $x \in \mathscr X$ we have $|f(B_T(x,y))| = |f(\langle T(x),y \rangle)| \geq f(|T(x)|)f(|y|)$ for some element $y \in \mathscr X$ with $\|y\|=1$ selected with respect to $T(x)$. However, by the positivity and invertibility of $T$ the inequality $\langle T(x),T(x) \rangle \geq \|T^{-1}\|^{-2} \langle x,x \rangle$ holds for any $x \in \mathscr X$. So we arrive at $|f(B_T(x,y))| \geq \|T^{-1}\|^{-1} f(|x|)f(|y|)$. The second condition of Theorem \ref{l2} can be verified in a similar way. 

Since for full self-dual Hilbert $C^*$-modules $\mathscr X$ the $C^*$-algebra of coefficients is always unital and every bounded module operator on $\mathscr X$ admits an adjoint operator we can consider the discussed sesquilinear forms in a different way, cf. \cite[Thm. 3.7, (i), (iv)]{Fr99}:  Any other $C^*$-valued inner product $\langle\cdot,\cdot\rangle_0$ on $\mathscr X$ inducing an equivalent norm to the given one can be expressed as $\langle x,y \rangle_0=\langle T(x),y \rangle$ for any $x,y \in \mathscr{X}$ and for an unique invertible positive bounded module operator $T$ on $\mathscr X$. So for $W^*$-algebras and for self-dual Hilbert $W^*$-modules over them we obtained a large class of examples.
\end{remark}

We finish this section giving a version of the Lax--Milgram Theorem for Hilbert spaces. 

\begin{corollary} 
Let $\mathscr{X}$ and $\mathscr{Y}$ be Hilbert spaces, let $\{\mathscr{X}_\lambda\}_{\lambda\in \Lambda}$ be a family of closed subspaces of $\mathscr{X}$, let $\{\mathscr{Y}_\lambda\}_{\lambda\in \Lambda}$ be an upwards directed family of closed subspaces of $\mathscr{Y}$ and let $\mathscr{V}=\cup \mathscr{Y}_\lambda$. Suppose that
\[
B:\mathscr{X}\times \mathscr{V}\to \mathbb{C}
\]
is a map satisfying:
\begin{itemize}
\item[(a)]$B_\lambda=B|_{\mathscr{X}_\lambda\times \mathscr{Y}_\lambda}$ is a bounded sesqulinear form for all $\lambda\in \Lambda$
\item[(b)]$B(.,v)$ is a bounded linear functional on $\mathscr{X}$ for all $v\in \mathscr{V}$
\item[(c)]there exists $c>0$ such that for any $\lambda\in \Lambda$ and for any $x\in \mathscr X_\lambda$ (resp. for any $y\in\mathscr Y_\lambda$) there exists a unit vector $y\in\mathscr Y_\lambda$ (resp. $x\in\mathscr X_\lambda$) such that
\[|B_\lambda(x,y))|\geq c\|x\|\,\|y\|\].
 \end{itemize}
 Then for each $\tau\in \mathscr{V}',$ there is $x\in \mathscr{X}$ such that $\|x\|\leq \frac{\|\tau\|}{c}$ and
 \[
 B(x,v)=\tau(v)
 \]
 for all $v\in \mathscr{V}$.
\end{corollary}

\section{The Lax--Milgram Theorem extended to Hilbert $C^*$-modules over compact $C^*$-algebras}

Suppose, $\mathscr{A}$ is a $C^*$-algebra of compact operators. Then we obtain the following analog of the Lax-Milgram theorem:

\begin{theorem}
Let $\mathscr{X}$ and $\mathscr{Y}$ be Hilbert $C^*$-modules over a $C^*$-algebra of compact operators. Suppose that
\[
B:\mathscr{X}\times \mathscr{Y}\to \mathscr{A}
\]
is a map satisfying:
\begin{itemize}
\item[(a)]$B$ is a bounded $\mathscr{A}$-sesqulinear form;
\item[(b)]there exists $c>0$ and $k>0$ such that for any $f\in \mathcal P \mathcal{S}(\mathscr{A})$ and any $x\in \mathscr{X}$ there exists $y\in\mathscr Y$ such that $\|y\|=1$, $f(|y|)\geq k$ and 
\[
|f(B(x,y))|\geq cf(|x|)f(|y|);
\]
\item[(c)]there exists $c>0$ and $k>0$ such that for any $f\in \mathcal P \mathcal{S}(\mathscr{A})$ and any $y\in \mathscr{Y}$ there exists $x\in\mathscr X$ such that $\|x\|=1$, $f(|x|)\geq k$ and 
\[
|f(B(x,y))|\geq cf(|x|)f(|y|).
\]
\end{itemize}
Then there exists a bounded adjointable invertible $\mathscr{A}$-linear operator $T: {\mathscr{X}} \to \mathscr{Y}$ such that $B(x,y) = \langle T(x),y \rangle$ for any $x \in \mathscr{X}$, any $y \in \mathscr{Y}$. 
\end{theorem}

\begin{proof}
Using strict topology arguments as in \cite{BaGu04,BaGu06}, the map $B(\cdot,\cdot)$ can be extended to a map $B_M: M(\mathscr{X}) \times M(\mathscr{Y}) \to M(\mathscr{A})$. Both $M(\mathscr{X})$, $M(\mathscr{Y})$ are self-dual Hilbert $W^*$-modules over the $W^*$-algebra $M(\mathscr{A})$. The set of pure states on $\mathscr A$ is weak dense in the set of pure states on $M({\mathscr A})$. Also, $\mathscr X$ and $\mathscr Y$ are strictly dense in $M({\mathscr X})$ and $M({\mathscr Y})$, respectively. So, if the two conditions (b) and (c) hold for elements from $\mathscr X$ and for pure states on $\mathscr A$ they hold for elements from $M({\mathscr X})$ and for pure states on $M({\mathscr A})$, too. Therefore, Corollary \ref{LM_Wstar} can be applied. We obtain $B_M(x,y) = \langle T_M(x),y \rangle$ for a bounded adjointable invertible $\mathscr{A}$-linear map $T_M: M({\mathscr{X}}) \to M({\mathscr{Y}})$, for any $x \in M({\mathscr{X}})$, any $y \in M({\mathscr{Y}})$. Restricting $B_M(.,.)$ (and respectively, $T_M$) back to $\mathscr{X}$ and $\mathscr{Y}$ we get $B(.,.)$ back, that is the image of $B(.,.)$ belongs to $\mathscr{A}$. Moreover, $T_M$ maps $\mathscr{X}$ into $\mathscr{Y}$ by the ideal properties of the Hilbert $\mathscr{A}$-modules in their multiplier modules and by \cite[Proposition 3.11]{Pa}. 
\end{proof}

Recall, that the Riesz theorem is not valid for certain Hilbert $C^*$-modules over $C^*$-algebras of compact operators, whereas the Lax--Milgram theorem turns out to be valid for all of them. Moreover, the theorem above can be considered as a Lax--Milgram theorem for a class of (not full, in general) non-self-dual, in general, Hilbert $W^*$-modules because $\mathscr A$-modules are always $M({\mathscr A})$-modules.



\begin{thebibliography}{99}

\bibitem{Arw} W. Arveson, \textit{An Invitation to $C^*$-algebras}, Springer, New York, 1976.

\bibitem{BaGu04} D. Baki\'c, B. Gulja\v{s}, \textit{Extensions of Hilbert $C^*$-modules, I}, Houston J. Math. \textbf{30} (2004), 537--558.

\bibitem{BaGu06} D. Baki\'c, B. Gulja\v{s}, \textit{On a class of module maps of Hilbert $C^*$-modules}, Math. Commun. \textbf{7} (2003), 177--192.

\bibitem{Brown} L. G. Brown, \textit{Stable isomorphism of hereditary subalgebras of $C^*$-algebras}, Pacific J. Math. \textbf{71} (1971), 335--348.

\bibitem{DY} D. Drivaliaris and N. Yannakakis, \textit{Generalizations of the Lax--Milgram theorem}, Bound. Value Probl. \textbf{2007}, Art. ID 87104, 9 pp.

\bibitem{Fr99} M. Frank, \textit{Geometrical aspects of Hilbert $C^*$-modules}, Positivity \textbf{3}(1999), 215-243.

\bibitem{Fr08} M. Frank, \textit{Characterizing $C^*$-algebras of compact operators by generic categorical properties of Hilbert $C^*$-modules}, 
J. K-Theory \textbf{2} ((2008), 453--462.

\bibitem{GV} C. Garetto and H. Vernaeve, \textit{Hilbert $\widetilde{\Bbb C}$-modules: structural properties and applications to variational problems}, Trans. Amer. Math. Soc. \textbf{363} (2011), no. 4, 2047--2090. 

\bibitem{KY} H. Kozono and T. Yanagisawa, \textit{Generalized Lax--Milgram theorem in Banach spaces and its application to the elliptic system of boundary value problems}, Manuscripta Math. \textbf{141} (2013), no. 3-4, 637--662. 

\bibitem{LAN} E. C. Lance, \textit{Hilbert $C^*$-modules. A Toolkit for Operator Algebraists}, London Mathematical Society Lecture Note Series, vol. 210, Cambridge University Press, Cambridge, 1995.

\bibitem{LM} P. D. Lax and A. N. Milgram, \textit{Parabolic equations}, Contributions to the theory of partial differential equations, pp. 167--190. Annals of Mathematics Studies, no. 33. Princeton University Press, Princeton, N. J., 1954.

\bibitem{LIN} H. Lin, \textit{Injective Hilbert $C^\ast$-modules}, Pacific J. Math. \textbf{154} (1992), no. 1, 131--164.

\bibitem{LKL} K. M. Levere, H. Kunze, and D. La Torre, \textit{A collage-based approach to solving inverse problems for second-order nonlinear parabolic PDEs}, J. Math. Anal. Appl. \textbf{406} (2013), no. 1, 120--133. 

\bibitem{Pa} W. L. Paschke, \textit{Inner Product modules over $B^*-$algebras}, Trans. Amer. Math. Soc. \textbf{182} (1972), 443--468.

\bibitem{MT} V. M. Manuilov and E. V. Troitsky, \textit{Hilbert $C^*$-modules}, In: Translations of Mathematical Monographs. 226, American Mathematical Society, Providence, RI, 2005.

\bibitem{R} G. J. Murphy, \textit{$C^*$-Algebras And Operator Theory}, Academic Press, INC, 1990.

\bibitem{RAM} S. Ramaswamy, \textit{The Lax--Milgram theorem for Banach spaces. II}, Proc. Japan Acad. Ser. A Math. Sci. \textbf{57} (1981), no. 1, 29--33. 

\bibitem{WR} W. Rudin, \textit{Functional Analysis},McGraw-Hill, INC, 1973

\bibitem{SAI} J. Saint Raymond, \textit{A generalization of Lax--Milgram's theorem}, Matematiche (Catania) \textbf{52} (1997), no. 1, 149--157 (1998).

\bibitem{Schw} J. Schweizer, \textit{A description of Hilbert $C^*$-modules in which all closed submodules are orthogonally closed}, Proc. Amer. Math. Soc. \textbf {127} (1999), 2123--2125.

\bibitem{RUI} G. M. Ruiz, \textit{A version of the Lax--Milgram theorem for locally convex spaces}, J. Convex Anal. \textbf{16} (2009), no. 3-4, 993--1002.

\end{thebibliography}
\end{document}